\documentclass[11pt]{article}

\usepackage{enumerate}
\usepackage{amsmath}
\usepackage{amssymb,latexsym}
\usepackage{amsthm}
\usepackage{color}

\usepackage{graphicx}

\usepackage{hyperref}

\usepackage{amscd}

\DeclareMathOperator{\Tr}{Tr}
\DeclareMathOperator{\N}{N}

\title{Classes and equivalence of linear sets in $\mathrm{PG}(1,q^n)$}
\author{Bence Csajb\'ok, Giuseppe Marino and Olga Polverino\thanks{\textcolor{black}{The
research  was supported by
Ministry for Education, University and Research of Italy MIUR (Project
PRIN 2012 "Geometrie di Galois e strutture di incidenza") and by the Italian National
Group for Algebraic and Geometric Structures and their Applications (GNSAGA
- INdAM).}}}
\date{}

\newcommand{\cA}{{\mathcal A}}
\newcommand{\cB}{{\mathcal B}}

\newcommand{\cC}{{\mathcal C}}

\newcommand{\cV}{{\mathcal V}}

\newcommand{\cM}{{\mathcal M}}

\newcommand{\cH}{{\mathcal H}}

\newcommand{\cQ}{{\mathcal Q}}

\newcommand{\F}{{\mathbb F}}

\newcommand{\V}{{\mathbb V}}

\newcommand{\la}{\langle}
\newcommand{\ra}{\rangle}

\newcommand{\ZG}{\mathcal{Z}(\mathrm{\Gamma L})}
\newcommand{\G}{\mathrm{\Gamma L}}

\newtheorem{theorem}{Theorem}[section]

\newtheorem{lemma}[theorem]{Lemma}
\newtheorem{corollary}[theorem]{Corollary}
\newtheorem{definition}[theorem]{Definition}
\newtheorem{proposition}[theorem]{Proposition}
\newtheorem{result}[theorem]{Result}
\newtheorem{example}[theorem]{Example}

\newtheorem{remark}[theorem]{Remark}

\DeclareMathOperator{\PG}{{PG}}

\begin{document}
\maketitle

\begin{abstract}
The equivalence problem of $\F_q$-linear sets of rank $n$ of $\PG(1,q^n)$
is investigated, also in terms of the associated variety, projecting configurations, $\F_q$-linear blocking sets of R\'edei type and MRD-codes.
\end{abstract}



\section{Introduction}
\label{sec:Intro}

Linear sets are natural generalizations of subgeometries.
Let $\Lambda=\PG(W,\F_{q^n})\allowbreak=\PG(r-1,q^n)$, where $W$ is a vector space of dimension $r$ over $\F_{q^n}$. A point set $L$ of $\Lambda$ is said to be an \emph{$\F_q$-linear set} of $\Lambda$ of rank $k$ if it is
defined by the non-zero vectors of a $k$-dimensional $\F_q$-vector subspace $U$ of $W$, i.e.
\[L=L_U=\{\la {\bf u} \ra_{\mathbb{F}_{q^n}} \colon {\bf u}\in U\setminus \{{\bf 0} \}\}.\]
The maximum field of linearity of an $\F_q$-linear set $L_U$ is $\F_{q^t}$ if $t$ is the largest integer such that $L_U$ is an $\F_{q^t}$-linear set.
In the recent years, starting from the paper \cite{Lu1999} by Lunardon, linear sets have been used to construct or characterize various objects in finite geometry, such as blocking sets and multiple blocking sets in finite projective spaces, two-intersection sets in finite projective spaces, translation spreads of the Cayley Generalized Hexagon, translation ovoids of polar spaces, semifield flocks and finite semifields. For a survey on linear sets we refer the reader to \cite{OP2010}, see also \cite{Lavrauw}.

One of the most natural questions about linear sets is their equivalence. Two linear sets $L_U$ and $L_V$ of $\PG(r-1,q^n)$ are said to be \emph{$\mathrm{P\Gamma L}$-equivalent} (or simply \emph{equivalent}) if there is an element $\varphi$ in $\mathrm{P\Gamma L}(r,q^n)$ such that $L_U^{\varphi} = L_V$. In the applications it is crucial to have methods to decide whether two linear sets are equivalent or not.
For $f\in \mathrm{\Gamma L}(r,q^n)$ we have $L_{U^{f}}=L_{U}^{\varphi_f}$, where $\varphi_f$ denotes the collineation of $\PG(W,\F_{q^n})$ induced by $f$. It follows that if $U$ and $V$ are $\F_q$-subspaces of $W$ belonging to the same orbit of
$\mathrm{\Gamma L}(r,q^n)$, then $L_U$ and $L_V$ are equivalent.
The above condition is only sufficient but not necessary to obtain equivalent linear sets.
This follows also from the fact that $\F_q$-subspaces of $W$ with different ranks can define the same linear set, for example $\F_q$-linear sets of $\PG(r-1,q^n)$ of rank $k\geq rn-n+1$ are all the same: they coincide with $\PG(r-1,q^n)$.
As it was showed recently in \cite{CSZ2015}, if $r=2$, then there exist $\F_q$-subspaces of $W$ of the same rank $n$ but on different orbits of $\Gamma \mathrm{L}(2,q^n)$ defining the same linear set of $\PG(1,q^n)$.

Suppose that $L_U^{\varphi_f}=L_V$ for some collineation, but there is no $\F_{q^n}$-semilinear map between $U$ and $V$.
Then the $\F_q$-subspaces $U^{f}$ and $V$ define the same linear set, but there is no invertible $\F_{q^n}$-semilinear map between them.
This observation motivates the following definition. An $\F_q$-linear set $L_U$ with maximum field of linearity $\F_q$ is called \emph{simple} if for each $\F_q$-subspace $V$ of $W$ with $\dim_q(U)=\dim_q(V)$, $L_U=L_V$ only if $U$ and $V$ are in the same orbit of $\Gamma \mathrm{L}(W,\F_{q^n})$.
Natural examples of simple linear sets are the subgeometries (cf. \cite[Theorem 2.6]{LV2013} and \cite[Section 25.5]{JWPH3}).
In \cite{BoPo2005} it was proved that $\F_q$-linear sets of rank $n+1$ of $\PG(2,q^n)$ admitting $(q+1)$-secants are simple.
This allowed the authors to translate the question of equivalence to the study of the orbits of the stabilizer of a subgeometry on subspaces and hence to obtain the complete classification of $\F_q$-linear blocking sets in $\PG(2,q^4)$.
Until now, the only known examples of non-simple linear sets are those of pseudoregulus type of $\PG(1,q^n)$ for $n\geq 5$ and $n\neq 6$, see \cite{CSZ2015}.

In this paper we focus on linear sets of rank $n$ of $\PG(1,q^n)$. Such linear sets are related to $\F_q$-linear blocking sets of R\'edei type, MRD-codes of size $q^{2n}$ with minimum rank distance $n-1$ and projections of subgeometries. We first introduce a method which can be used to find non-simple linear sets of rank $n$ of $\PG(1,q^n)$. Let $L_U$ be a linear set of rank $n$ of $\PG(W,\F_{q^n})=\PG(1,q^n)$ and let $\beta$ be a non-degenerate alternating form of $W$. Denote by $\perp$ the orthogonal complement map induced by $\Tr_{q^n/q} \circ \beta$ on $W$ (considered as an $\F_q$-vector space). Then $U$ and $U^\perp$ defines the same linear set (cf. Result \ref{rem:dualsympl}) and if $U$ and $U^{\perp}$ lie on different orbits of $\mathrm{\Gamma L}(W,\F_{q^n})$, then $L_U$ is non-simple.
Using this approach we show that there are non-simple linear sets of rank $n$ of $\PG(1,q^n)$ for $n\geq 5$, not of pseudoregulus type (cf. Proposition \ref{nemsimple}). Contrary to what we expected initially, simple linear sets are harder to find. We prove that the linear set of $\PG(1,q^n)$ defined by the trace function is simple (cf. Theorem \ref{THMtrace}). We also show that linear sets of rank $n$ of $\PG(1,q^n)$ are simple for $n\leq 4$
(cf. Theorem \ref{n4simple}).

Moreover, in $\PG(1,q^n)$ we extend the definition of simple linear sets and introduce the $\ZG$-class and the $\G$-class for linear sets of rank $n$.
In Section \ref{aspects} we point out the meaning of these classes in terms of equivalence of the associated blocking sets, MRD-codes and projecting configurations.


\section{Definitions and preliminary results}
\label{Prelim}
\subsection{Dual linear sets with respect to a symplectic polarity of a line}
\label{subdual}

For $\alpha\in \F_{q^n}$ and a divisor $h$ of $n$ we will denote by $\Tr_{q^n/q^h}(\alpha)$ the trace of $\alpha$ over the subfield $\F_{q^h}$, that is, $\Tr_{q^n/q^h}(\alpha)=\alpha+\alpha^{q^h}+\ldots+\alpha^{q^{n-h}}$.
By $\N_{q^n/q^h}(\alpha)$ we will denote the norm of $\alpha$ over the subfield $\F_{q^h}$, that is,
$\N_{q^n/q^h}(\alpha)=\alpha^{1+q^h+\ldots+q^{n-h}}$. Since in the paper we will use only norms over $\F_q$, the function $\N_{q^n/q}$ will be denoted simply by $\N$.


Starting from a linear set $L_U$ and using a polarity $\tau$ of the space it is always possible to construct another linear set, which is called {\it dual
linear set of $L_U$ with respect to  the polarity} $\tau$ (see \cite{OP2010}). In particular, let $L_U$ be an $\F_q$--linear set of rank $n$ of a line $\PG(W,\F_{q^n})$ and let $\beta: W\times W\longrightarrow \F_{q^n}$ be a non-degenerate reflexive $\F_{q^n}$--sesquilinear form on the 2-dimensional vector space $W$ over $\F_{q^n}$ determining a polarity $\tau$. The map $\Tr_{q^n/q}\circ \beta$ is a
non-degenerate reflexive $\F_q$--sesquilinear form on $W$, when
$W$ is  regarded as a $2n$-dimensional vector space  over
$\F_{q}$. Let $\perp_\beta$ and $\perp'_\beta$ be the orthogonal complement
maps defined by $\beta$ and $\Tr_{q^n/q}\circ \beta$ on the
lattices  of the $\F_{q^n}$-subspaces and $\F_{q}$-subspaces of
$W$, respectively. The dual linear set of $L_U$ with respect to the polarity $\tau$ is the $\F_q$--linear set of rank $n$ of $\PG(W,\F_{q^n})$ defined by the orthogonal complement
$U^{\perp'_\beta}$ and it will be denoted by $L^{^\tau}_U$. Also, up to projectively equivalence, such a linear
set does not depend on $\tau$.

\bigskip
For a point $P=\la {\bf z} \ra_{\F_{q^n}}\in \PG(W,\F_{q^n})$ the \emph{weight} of $P$ with respect to the linear set $L_U$ is $w_{L_U}(P):=\dim_q(\la {\bf z} \ra_{\F_{q^n}} \cap U)$. Note that when $P\in L_U$, then the weight depends on the subspace $U$ and not only on the set of points of $L_U$. It can happen that for two $\F_q$-subspaces $U$ and $V$ of $W$ we have $L_U=L_V$ with $w_{L_U}(P)\neq w_{L_V}(P)$. When we write ``the weight of $P\in L_U$'', then we always mean $w_{L_U}(P)$ and hence when we speak about the weight of a point, we will never omit the subscript.

\begin{result}
\label{rem:dualsympl}
{\rm From \cite[Property 2.6]{OP2010} (with $r=2$, $s=1$ and $t=n$) it can be easily seen that if $L_U$ is an $\F_q$--linear set of rank $n$ of a line $\PG(W,\F_{q^n})$ and $L^{^\tau}_U$ is its dual linear set with respect to a polarity $\tau$, then $w_{L^{\tau}_U}(P^\tau)=w_{L_U}(P)$ for each point $P\in \PG(W,\F_{q^n})$. If $\tau$ is a symplectic polarity of a line $PG(W,\F_{q^n})$, then $P^\tau=P$ and hence $L_U=L_U^\tau=L_{U^{\perp'_\beta}}$.}
\end{result}

\subsection{\texorpdfstring{$\F_q$-linear sets of $\PG(1,q^n)$ of class $r$}{Fq-linear sets of PG(1,qn) of class r}}
\label{subtitle}

In this paper we investigate the equivalence of $\F_q$-linear sets of rank $n$ of the projective line $\PG(W,\F_{q^n})=\PG(1,q^n)$.
As we have seen in the introduction, two $\F_q$-linear sets $L_U$ and $L_V$ of rank $n$ of $\PG(1,q^n)$ are equivalent if there is an element $\varphi_f$ in $\mathrm{P\Gamma L}(2,q^n)$ such that $L_U^{\varphi_f} = L_{U^{f}}=L_V$, where $f\in \mathrm{\Gamma L}(W,\F_{q^n})$ is the semilinear map inducing $\varphi_{f}$. Hence the first step is to determine the $\F_{q}$-vector subspaces of $W$ defining the same linear set.
This motivates the definition of the $\mathcal{Z}(\mathrm{\Gamma L})$-class and $\mathrm{\Gamma L}$-class of a linear set $L_U$ of $\PG(1,q^n)$ (cf. Definitions \ref{ZGL-class} and \ref{GL-class}). The next proposition relies on the characterization of functions over $\F_q$ determining few directions.
It states that the $\F_q$-rank of $L_U$ of $\PG(1,q^n)$ is uniquely defined when the maximum field of linearity of $L_U$ is $\F_q$.
This will allow us to state our definitions and results without further conditions on the rank of the corresponding $\F_q$-subspaces.

\begin{proposition}
\label{rankisgood}
Let $L_U$ be an $\F_q$-linear set of $\PG(W,\F_{q^n})=\PG(1,q^n)$ of rank $n$.
The maximum field of linearity of $L_U$ is $\F_{q^d}$, where
\[d=\min\{ w_{L_U}(P) \colon P\in L_U\}.\]
If the maximum field of linearity of $L_U$ is $\F_q$, then the rank of $L_U$ as an $\F_q$-linear set is uniquely defined, i.e. for each $\F_q$-subspace $V$ of $W$ if $L_U=L_V$, then $\dim_{q}(V)=n$.
\end{proposition}
\begin{proof}
First assume that $\la(0,1)\ra_{\F_{q^n}}\notin L_U$, i.e. $U=\{(x,f(x)) \colon x\in \F_{q^n}\}$ for some $q$-polynomial $f$ over $\F_{q^n}$.

Consider the following map, $U \rightarrow \PG(2,q^n) \colon (x,f(x))\mapsto \la(x,f(x),1)\ra_{\F_{q^n}}$.
We will call this $q$-set of $\PG(2,q^n)$ the graph of $f$ and we will denote it by $G_f$. Let $X_0$, $X_1$, $X_2$ denote the coordinate functions in $\PG(2,q^n)$ and consider the line $X_2=0$ as the line at infinity, denoted by $\ell_{\infty}$. The points of $\ell_\infty$ are called directions, denoted by $(m):=\la(1,m,0)\ra_{\F_{q^n}}$ and by $(\infty):=\la(0,1,0)\ra_{\F_{q^n}}$.
The set of directions determined by $f$ is
\[D_f:=\left\{  \left(\frac{f(x)-f(y)}{x-y}\right) \colon x,y\in \F_{q^n},\,x\neq y\right\}=
\left\{ \left( \frac{f(z)}{z} \right) \colon z\in \F_{q^n}^*\right\}.\]
It follows that $\la(x,f(x))\ra_{q^n}\mapsto\la(x,f(x),0)\ra_{\F_{q^n}}$ is a bijection between
the point set of  $L_U$ and the set of directions determined by $f$.
The point $P_m:=\la(1,m)\ra_{\F_{q^n}}$ is mapped to the direction $(m)$.

For each line $\ell$ through $(m)$ if $\ell$ meets the graph of $f$, then it meets it in $q^t$ points, where
$t=w_{L_{U}}(P_m)$. Indeed, suppose that $\ell$ meets the graph of $f$ in $\la(x_0,f(x_0),1)\ra_{\F_{q^n}}$.
To obtain the number of the other points of $\ell \cap G_f$ we have to count
\[\left|\left\{ x \in \F_{q^n}\setminus \{x_0\} \colon \frac{f(x)-f(x_0)}{x-x_0}=m  \right\}\right|=
\left|\left\{ z \in \F_{q^n}^* \colon \frac{f(z)}{z}=m  \right\}\right|,\]
which is $q^t-1$.

Let $d=\min\{ w_{L_U}(P) \colon P\in L_U\}$. If $q=p^e$, $p$ prime, then $p^{de}$ is the largest $p$-power such that every line meets the graph of $f$ in a multiple of $p^{de}$ points. Then a result on the number of direction determined by functions over $\F_q$ due to Ball, Blokhuis, Brouwer, Storme and Sz\H onyi \cite{BBBSSz}, and Ball \cite{B2003} yields that either $d=n$ and $f(x)=\lambda x$ for some $\lambda\in \F_{q^n}$, or $\F_{q^d}$ is a subfield of $\F_{q^n}$ and
\begin{equation}
\label{dirp}
q^{n-d}+1 \leq |D_f| \leq  \frac{q^n-1}{q^d-1}.
\end{equation}
Moreover, if $q^d>2$, then $f$ is $\F_{q^d}$-linear.
In our case we already know that $f$ is $\F_q$-linear, so even in the case $q^d=2$ it follows that
$U$ is an $\F_{q^d}$-subspace of $W$ and hence $L_U$ is an $\F_{q^d}$-linear set. We show that $\F_{q^d}$ is the maximum field of linearity of $L_U$. Suppose, contrary to our claim, that $L_U$ is $\F_{q^r}$-linear of rank $z$ for some $r>d$. Then $L_U$ is also $\F_q$-linear of rank $rz$. It follows that $rz\leq n$ since otherwise $L_U=\PG(1,q^n)$. Then for the size of $L_U$ we get $|L_U|\leq (q^{rz}-1)/(q^r-1)\leq (q^n-1)/(q^r-1)$.
To get a contradiction, we show that this is less than $q^{n-d}+1$, which is the lower bound obtain for $|L_U|$ in \eqref{dirp}. After rearranging we get
\[\frac{q^n-1}{q^r-1}< q^{n-d}+1 \Leftrightarrow q^{n-d}(q^d+1)< (q^{n-d}+1)q^r.\]
The latter inequality always holds because of $r\geq d+1$. This contradiction shows $r=n$.

Now suppose that $\F_q$ is the maximum field of linearity of $L_U$ and let $V$ be an $r$-dimensional $\F_q$-subspace of $W$ such that $L_U=L_V$. We cannot have $r>n$ since $L_U\neq \PG(1,q^n)$. Suppose, contrary to our claim, that
$r\leq n-1$. Then $|L_U|\leq (q^{n-1}-1)/(q-1)$ contradicting \eqref{dirp} which gives $q^{n-1}+1 \leq |L_U|$.

Now suppose that $\la ( 0,1 ) \ra_{\F_{q^n}}\in L_U$. After a suitable projectivity $\varphi_f$ we have $\la(0,1)\ra_{\F_{q^n}}\notin L_{U^f}$.
Of course the maximum field of linearity of $L_U$ and $L_{U^f}$ coincide and for each point $P$ of $L_{U}$ we have $w_{L_U}(P)=w_{L_{U^f}}(P^{\varphi_f})$.
Hence the first part of the theorem follows. The second part also follows easily since $L_U=L_V$ with $\dim_q(U)\neq \dim_q(V)$ would yield
$L_{U^f}=L_{V^f}$ with $\dim_{q}(U^f)\neq \dim_q(V^f)$, a contradiction.
\end{proof}

Now we can give the following definitions of classes of an $\F_q$-linear set of a line.

\begin{definition}
\label{ZGL-class}
Let $L_U$ be an $\F_q$-linear set of $\PG(W,\F_{q^n})=\PG(1,q^n)$ of rank $n$ with maximum field of linearity $\F_q$.
We say that $L_U$ is of $\ZG$-class $r$ if $r$ is the largest integer such that there exist
$\F_q$-subspaces $U_1,U_2,\ldots,U_r$ of $W$ with $L_{U_i}=L_U$ for $i\in\{1,2,\ldots, r\}$ and
$U_i\neq \lambda U_j$ for each $\lambda\in \F_{q^n}^*$ and for each $i\neq j$, $i,j\in\{1,2,\ldots,r\}$.
\end{definition}

\begin{definition}
\label{GL-class}
Let $L_U$ be an $\F_q$-linear set of $\PG(W,\F_{q^n})=\PG(1,q^n)$ of rank $n$ with maximum field of linearity $\F_q$.
We say that $L_U$ is of $\G$-class $s$ if $s$ is the largest integer such that there exist $\F_q$-subspaces $U_1,U_2,\ldots,U_s$ of $W$ with $L_{U_i}=L_U$ for $i\in\{1,2,\ldots, s\}$ and there is no $f\in \mathrm{\Gamma L}(2,q^n)$ such that
$U_i=U_j^f$ for each $i\neq j$, $i,j\in\{1,2,\ldots,s\}$.
\end{definition}

Simple linear sets (cf. Section \ref{sec:Intro}) of $\PG(1,q^n)$ are exactly those of $\G$-class one.
The next propositions are easy to show.

\begin{proposition}
\label{invar0}
Let $L_U$ be an $\F_q$-linear set of $\PG(W,\F_{q^n})=\PG(1,q^n)$ of rank $n$ with maximum field of linearity $\F_q$ and let $P$ be a point of $\PG(1,q^n)$.
Then for each $f\in \Gamma \mathrm{L}(2,q^n)$ we have $w_{L_U}(P)=w_{L_{U^f}}(P^{\varphi_f})$.\qed
\end{proposition}

\begin{proposition}
\label{invar}
Let $L_U$ be an $\F_q$-linear set of $\PG(W,\F_{q^n})=\PG(1,q^n)$ of rank $n$ with maximum field of linearity $\F_q$ and let
$\varphi$ be a collineation of $\PG(W,\F_{q^n})$.
Then $L_U$ and $L_U^{\varphi}$ have the same $\ZG$-class and $\G$-class. \qed
\end{proposition}

\begin{remark}
Let $L_U$ be an $\F_q$-linear set of rank $n$ of $\PG(1,q^n)$ with $\G$-class $s$ and let $U_1, U_2, \ldots, U_s$ be $\F_q$-subspaces belonging to different orbits of $\Gamma \mathrm{L}(2,q^n)$ and defining $L_U$.
The $\mathrm{P\Gamma L}(2,q^n)$-orbit of $L_U$ is the set
\[
\bigcup_{i=1}^s \{ L_{U_i^f} \colon f \in \mathrm{\Gamma L}(2,q^n)\}.
\]
\end{remark}

\section{\texorpdfstring{Examples of simple and non-simple linear sets of $\PG(1,q^n)$}{Examples of simple and non-simple linear sets of PG(1,qn)}}

Let  $\V=\mathbb F_{q^n} \times \mathbb F_{q^n}$ and let $L_U$ be an $\F_q$--linear set of rank $n$ of $\PG(1,q^n)=\PG(\V,\F_{q^n})$. We can always assume (up to a projectivity) that $L_U$ does not contain the point $\la(0,1)\ra_{\F_{q^n}}$. Then $U=U_f=\{(x,f(x)) \colon x\in \F_{q^n}\}$, for some $q$-polynomial $f(x)=\sum_{i=0}^{n-1} a_i x^{q^i}$ over $\F_{q^n}$. For the sake of simplicity we will write $L_f$ instead of $L_{U_f}$ to denote the linear set defined by $U_f$.

According to Result \ref{rem:dualsympl} and using the same notations as in Section \ref{subdual} if $L_U$ is an $\F_q$-linear set of rank $n$ of $\PG(1,q^n)$ and $\tau$ is a symplectic polarity,
then $U^{\perp'_\beta}$ defines the same linear set as $U$. Since in general $U^{\perp'_\beta}$ and $U$ are not equivalent under the action of the group $\mathrm{\Gamma L}(2,q^n)$, simple linear sets of a line are harder to find.

Consider the non-degenerate symmetric bilinear form of $\mathbb F_{q^n}$
over $\F_q$ defined by the following rule
\begin{equation}\label{form:angle}
<x,y>:=\Tr_{q^n/q}(xy).
\end{equation}

\noindent Then the {\it adjoint map} $\hat f$ of an $\F_q$-linear map $f(x)=\sum_{i=0}^{n-1}a_i x^{q^i}$ of $\F_{q^n}$ (with respect to the bilinear form $\langle,\rangle$) is
\begin{equation}
\hat f(x):=\sum_{i=0}^{n-1}a_i^{q^{n-i}} x^{q^{n-i}}.
\end{equation}

Let $\eta: \V \times \V \longrightarrow \F_{q^{n}}$ be the  non-degenerate alternating bilinear  form of
$\V$ defined by $\eta((x,y), (u,v))=xv-yu$. Then $\eta$ induces a symplectic polarity on the line $\PG(\V,\F_{q^n})$ and
\begin{equation}\label{form:perp1}
\eta'((x,y), (u,v))=\Tr_{q^n/q}(\eta((x,y), (u,v)))\end{equation}
is a non-degenerate alternating bilinear form on $\V$, when $\V$ is regarded as a $2n$-dimensional vector space over $\F_{q}$.
We will always denote in the paper by $\perp$ and $\perp'$ the orthogonal complement maps defined by $\eta$ and $\eta'$ on the lattices of
the $\F_{q^n}$-subspaces and the $\F_{q}$-subspaces of $\V$, respectively. Direct calculation shows
that \begin{equation}\label{form:hatf} U_{f}^{\perp '}=U_{\hat f}.\end{equation}

Result \ref{rem:dualsympl} and (\ref{form:hatf}) allow us to slightly reformulate \cite[Lemma 2.6]{BGMP2015}.

\begin{lemma}[{\cite{BGMP2015}}]
\label{lem2} Let $L_f=\{\langle(x,f(x))\rangle_{\F_{q^n}} \colon x\in \F_{q^n}^*\}$ be an $\F_q$--linear set of $\PG(1,q^n)$ of rank $n$, with $f(x)$ a $q$-polynomial over $\F_{q^n}$, and let $\hat f$ be the adjoint of $f$ with respect to the bilinear form (\ref{form:angle}). Then for each point $P\in \PG(1,q^n)$ we have $w_{L_{f}}(P)=w_{L_{\hat f}}(P)$. In particular,
$L_{f}=L_{\hat f}$ and the maps defined by $f(x)/x$ and $\hat f(x)/x$  have
the same image.
\end{lemma}

\begin{lemma}
\label{lem3}
Let $\varphi$ be an $\F_q$-linear map of $\F_{q^{n}}$ and for $\lambda\in \F_{q^n}^*$ let $\varphi_{\lambda}$ denote the $\F_q$-linear map: $x \mapsto \varphi(\lambda x)/\lambda$.
Then for each point $P\in \PG(1,q^n)$ we have $w_{L_{\varphi}}(P)=w_{L_{\varphi_{\lambda}}}(P)$. In particular,
$L_{\varphi}=L_{\varphi_{\lambda}}$.
\end{lemma}
\begin{proof}
The statements follow from $\lambda U_{\varphi_{\lambda}}=U_{\varphi}$.
\end{proof}

\begin{remark}
The results of Lemmas \ref{lem2} and \ref{lem3} can also be obtained via Dickson matrices.
For a $q$-polynomial $f$ let $D_f$ denote the Dickson matrix associated with $f$.
When $f(x)=\lambda x$ for some $\lambda\in \F_{q^n}$ we will simply write $D_{\lambda}$.
We will denote the point $\la(1,\lambda)\ra_{q^n}$ by $P_{\lambda}$.

Transposition preserves the rank of matrices and $D_f^T=D_{\hat f}$, $D_\lambda^T=D_\lambda$. It follows that
\[\dim_q\ker(D_f-D_\lambda)=\dim_q\ker(D_f-D_\lambda)^T=\dim_q\ker(D_{\hat f}-D_\lambda),\]
and hence for each $\lambda\in \F_{q^n}$ we have $w_{L_f}(P_\lambda)=w_{L_{\hat f}}(P_\lambda)$.

Let $f_{\mu}(x)=f(x\mu)/\mu$. It is easy to see that $D_{1/\mu}D_fD_{\mu}=D_{f_\mu}$ and
\[\dim_q\ker(D_f-D_\lambda)=\dim_q\ker D_{1/\mu}(D_f-D_\lambda)D_{\mu}=\dim_q\ker(D_{f_{\mu}}-D_\lambda),\]
and hence $w_{L_f}(P_{\lambda})=w_{L_{f_{\mu}}}(P_{\lambda})$ for each $\lambda\in \F_{q^n}$.
\end{remark}

From the previous arguments it follows that linear sets $L_f$ with $f(x)=\hat{f}(x)$ are good candidates for being simple. In the next section we show that the trace function, which has the previous property, defines a simple linear set. We are going to use the following lemmas which will also be useful later.

\begin{lemma}
\label{lmain}
Let $f$ and $g$ be two linearized polynomials. If $L_f=L_g$, then for each positive integer $d$ the following holds
\[\sum_{x\in \F_{q^n}^*} \left(\frac{f(x)}{x}\right)^d = \sum_{x\in \F_{q^n}^*} \left(\frac{g(x)}{x}\right)^d.\]
\end{lemma}
\begin{proof}
If $L_f=L_g=:L$, then $\{f(x)/x \colon x\in \F_{q^n}^*\}=\{g(x)/x \colon x\in \F_{q^n}^*\}=:H$.
For each $h\in H$ we have $|\{ x \colon f(x)/x=h\}|=q^i-1$, where $i$ is the weight of the point $\la(1,h)\ra_{q^n}\in L$ w.r.t. $U_f$,
and similarly $|\{ x \colon g(x)/x=h\}|=q^j-1$, where $j$ is the weight of the point $\la(1,h)\ra_{q^n}\in L$ w.r.t. $U_g$.
Because of the characteristic of $\F_{q^n}$, we obtain:
\[\sum_{x\in \F_{q^n}^*} \left(\frac{f(x)}{x}\right)^d = - \sum_{h\in H} h^d =\sum_{x\in F_{q^n}^*} \left(\frac{g(x)}{x}\right)^d.\]
\end{proof}

\begin{lemma}[Folklore]
\label{lfolk}
For any prime power $q$ and integer $d$ we have $\sum_{x\in \F_{q}^*} x^d=-1$ if $q-1 \mid d$ and
$\sum_{x\in \F_{q}^*} x^d=0$ otherwise.
\end{lemma}

\begin{lemma}
\label{mainLEMMA}
Let $f(x)=\sum_{i=0}^{n-1}a_i x^{q^i}$ and $g(x)=\sum_{i=0}^{n-1}b_i x^{q^i}$ be two $q$-polynomials over $\F_{q^n}$, such that $L_f=L_g$.
Then
\begin{equation}
\label{a0ba0}
a_0=b_0,
\end{equation}
and for $k=1,2,\ldots,n-1$ it holds that
\begin{equation}
\label{e1}
a_ka_{n-k}^{q^k}=b_kb_{n-k}^{q^k},
\end{equation}
for $k=2,3,\ldots,n-1$ it holds that
\begin{equation}
\label{e2}
a_1a_{k-1}^q a_{n-k}^{q^k}+a_ka_{n-1}^qa_{n-k+1}^{q^k}=b_1b_{k-1}^q b_{n-k}^{q^k}+b_kb_{n-1}^qb_{n-k+1}^{q^k}.
\end{equation}
\end{lemma}
\begin{proof}
We are going to use Lemma \ref{lfolk} together with Lemma \ref{lmain} with different choices of $d$.

With $d=1$ we have
\[\sum_{x\in \F_{q^n}^*} \sum_{i=0}^{n-1} a_i x^{q^i-1} = \sum_{x\in \F_{q^n}^*} \sum_{i=0}^{n-1} b_i x^{q^i-1},\]
and hence
\[\sum_{i=0}^{n-1} a_i \sum_{x\in \F_{q^n}^*} x^{q^i-1} = \sum_{i=0}^{n-1} b_i \sum_{x\in \F_{q^n}^*} x^{q^i-1}.\]
Since $q^n-1$ cannot divide $q^i-1$ with $i=1,2,\ldots,n-1$, $a_0=b_0=:c$ follows. Let $\varphi$ denotes the $\F_q$-linear map which fixes $(0,1)$ and maps $(1,0)$ to $(1,-c)$. Then $U_f^\varphi=U_{f'}$ and $U_g^{\varphi}=U_{g'}$ with $f'=\sum_{i=1}^{n-1} a_i x^{q^i}$, $g'=\sum_{i=1}^{n-1} b_i x^{q^i}$ and of course with $L_{f'}=L_{g'}$. It follows that we may assume $c=0$.

First we show that \eqref{e1} holds.
With $d=q^k+1$, $1\leq k \leq n-1$ we obtain
\[\sum_{1\leq i,j \leq n-1} a_ia_j^{q^k} \sum_{x\in \F_{q^n}^*} x^{q^i-1+q^{j+k}-q^k}=\sum_{1\leq i,j \leq n-1} b_ib_j^{q^k} \sum_{x\in \F_{q^n}^*} x^{q^i-1+q^{j+k}-q^k}.\]
$\sum_{x\in \F_{q^n}^*} x^{q^i-1+q^{j+k}-q^k}=-1$ if and only if $q^i+q^{j+k}\equiv q^k+1 \pmod {q^n-1}$, and zero otherwise.
Suppose that the former case holds.

First consider $j+k \leq n-1$.
Then $q^i+q^{j+k} \leq q^{n-1}+q^{n-1}<q^k+1+2(q^n-1)$ hence one of the following holds.
    \begin{itemize}
        \item If $q^i+q^{j+k}=q^k+1$, then the right hand side is not divisible by $q$, a contradiction.
        \item If $q^i+q^{j+k}=q^k+1+(q^n-1)=q^n+q^k$, then $j+k=n$, a contradiction.
     \end{itemize}

Now consider the case $j+k \geq n$.
Then $q^i+q^{j+k} \equiv q^i+q^{j+k-n} \equiv q^k+1 \pmod {q^n-1}$.
Since $j+k \leq 2(n-1)$, we have $q^i+q^{j+k-n} \leq q^{n-1}+q^{n-2}<q^k+1+2(q^n-1)$, hence one of the following holds.
    \begin{itemize}
        \item If $q^i+q^{j+k-n}= q^k+1$, then $j+k=n$ and $i=k$.
        \item If $q^i+q^{j+k-n}=q^k+1+(q^n-1)=q^n+q^k$, then there is no solution since $j+k-n \notin \{k,n\}$.
    \end{itemize}

Hence \eqref{e1} follows.
Now we show that \eqref{e2} also holds.
Note that in this case $n\geq 3$, otherwise there is no $k$ with $2\leq k \leq n-1$.
With $d=q^k+q+1$, we obtain
\[\sum_{1\leq i,j,m \leq n-1} a_ia_j^qa_m^{q^k} \sum_{x\in \F_{q^n}^*} x^{q^i-1+q^{j+1}-q+q^{m+k}-q^k}=\]
\[\sum_{1\leq i,j,m \leq n-1} b_ib_j^qb_m^{q^k} \sum_{x\in \F_{q^n}^*} x^{q^i-1+q^{j+1}-q+q^{m+k}-q^k}.\]
$\sum_{x\in \F_{q^n}^*} x^{q^i-1+q^{j+1}-q+q^{m+k}-q^k}=-1$ if and only if $q^i+q^{j+1}+q^{m+k}\equiv q^k+q+1 \pmod {q^n-1}$, and zero otherwise.
Suppose that the former case holds.

First consider $m+k \leq n-1$.
Then $q^i+q^{j+1}+q^{m+k} \leq q^{n-1}+q^n+q^{n-1}<q^k+q+1+2(q^n-1)$ hence one of the following holds.
    \begin{itemize}
        \item If $q^i+q^{j+1}+q^{m+k}= q^k+q+1$, then the right hand side is not divisible by $q$, a contradiction.
        \item If $q^i+q^{j+1}+q^{m+k}=q^k+q+1+(q^n-1)=q^n+q^k+q$, then $m+k=n$, $j+1=k$ and $i=1$, a contradiction.
     \end{itemize}

Now consider the case $m+k \geq n$.
Then $q^i+q^{j+1}+q^{m+k} \equiv q^i+q^{j+1}+q^{m+k-n} \equiv q^k+q+1 \pmod {q^n-1}$.
We have $q^i+q^{j+1}+q^{m+k-n} \leq q^{n-1}+q^n+q^{n-2}<q^k+q+1+2(q^n-1)$ hence one of the following holds.
    \begin{itemize}
        \item If $q^i+q^{j+1}+q^{m+k-n}= q^k+q+1$, then $j+1=k$, $i=1$ and $m+k=n$.
        \item If $q^i+q^{j+1}+q^{m+k-n}=q^k+q+1+(q^n-1)=q^n+q^k+q$, then $j+1=n$, $i=k$ and $m+k=n+1$.
    \end{itemize}
This concludes the proof.
\end{proof}

\subsection{Linear sets defined by the trace function}
\label{trace}

We show that there exist at least one simple $\F_q$-linear set in $\PG(1,q^n)$ for each $q$ and $n$. Let
$V=\{(x,\Tr_{q^n/q}(x)) \colon x\in \F_{q^n}\}$. We show that $L_U=L_V$ occurs for an $\F_q$-subspace $U$ of $W$ if and only if $V=\lambda U$ for some $\lambda\in \F_{q^n}^*$, i.e. $L_V$ is of $\ZG$-class one. For the special case when $L_U$ has a point of weight $n-1$ see also {\cite[Theorem 2.3]{MV2016}}.

\begin{theorem}
\label{THMtrace}
The $\F_q$-subspace $U_f=\{(x,f(x))\colon x\in \F_{q^n}\}$ defines the same linear set of $\PG(1,q^n)$ as the $\F_q$-subspace  $V=\{(x,\Tr_{q^n/q}(x)) \colon x\in \F_{q^n}\}$ if and only if $\lambda U_f=V$ for some $\lambda\in \F_{q^n}^*$, i.e. $L_V$ is simple.
\end{theorem}
\begin{proof}
Let $f(x)=\sum_{i=0}^{n-1}a_i x^{q^i}$.
We are going to use Lemma \ref{mainLEMMA} with $g(x)=\Tr_{q^n/q}(x)$. The coefficients $b_0,b_1, \ldots,b_{n-1}$ of $g(x)$ are 1, hence
$a_0=1$, and for $k=1,2,\ldots,n-1$
\begin{equation}
\label{f1}
a_k a_{n-k}^{q^k}=1,
\end{equation}
for $k=2,3,\ldots,n-1$
\begin{equation}
\label{f2}
a_1a_{k-1}^qa_{n-k}^{q^k}+a_ka_{n-1}^qa_{n-k+1}^{q^k}=2.
\end{equation}
Note that \eqref{f1} implies $a_i \neq 0$ for $i=1,2,\ldots,n-1$.
First we prove
\begin{equation}
\label{ind}
a_i=a_1^{1+q+\ldots+q^{i-1}}
\end{equation}
by induction on $i$ for each $0< i < n $.
The assertion holds for $i=1$. Suppose that it holds for some integer $i-1$ with $1< i < n$. We prove that it also holds for $i$.
Then \eqref{f2} with $k=i$ gives
\begin{equation}
\label{kell}
a_1a_{i-1}^qa_{n-i}^{q^i}+a_i a_{n-1}^q a_{n-i+1}^{q^i}=2.
\end{equation}
Also, \eqref{f1} with $k=i$, $k=i-1$ and $k=1$, respectively, gives
\[a_{n-i}^{q^i}=1/a_i,\]
\[a_{n-i+1}^{q^i}=1/a_{i-1}^q,\]
\[a_{n-1}^{q}=1/a_1.\]
Then \eqref{kell} gives
\begin{equation}
a_1a_{i-1}^q/a_i+a_i /\left(a_1 a_{i-1}^q\right)=2.
\end{equation}
It follows that $a_1a_{i-1}^q/a_i=1$ and hence the induction hypothesis on $a_{i-1}$ yields $a_i=a_1^{1+q+\ldots+q^{i-1}}$.

Finally we show $\N(a_1)=1$.
First consider $n$ even. Then \eqref{f1} with $k=n/2$ gives $a_{n/2}^{q^{n/2}+1}=1$. Applying \eqref{ind} yields $\N(a_1)=1$.
If $n$ is odd, then \eqref{f1} with $k=(n-1)/2$ gives $a_{(n-1)/2}a_{(n+1)/2}^{q^{(n-1)/2}}=1$.
Applying \eqref{ind} yields $\N(a_1)=1$.
It follows that $a_1=\lambda^{q-1}$ for some $\lambda\in \F_{q^n}^*$ and hence
$f(x)=\sum_{i=0}^{n-1}\lambda^{q^i-1} x^{q^i}$. Then $\lambda U_f = \{(x,\Tr_{q^n/q}(x)) \colon x\in \F_{q^n}^*\}$.
\end{proof}

\subsection{Non-simple linear sets}
\label{nonsimple}

So far, the only known non-simple linear sets of $\PG(1,q^n)$ are those of pseudoregulus type when $n=5$, or $n>6$, see Remark \ref{nonsimpleex}. Now we want to show that $\F_q$-linear sets $L_f$ of $\PG(1,q^n)$ introduced by Lunardon and Polverino, which are not of pseudoregulus type (\cite[Theorems 2 and 3]{LP2001}, are non-simple as well. Let start by proving the following preliminary result.

\begin{proposition}
\label{eqn0}
Let $f(x)=\sum_{i=0}^{n-1}a_ix^{q^i}$. There is an $\F_{q^n}$-semilinear
map between $U_f$ and $U_{\hat f}$ if and only if the following system of $n$ equations has a solution $A,B,C,D \in \F_{q^n}$, $AD-BC\neq 0$, $\sigma = p^k$:
\[C + Da_0^{\sigma}-a_0A=\sum_{i=0}^{n-1} (Ba_ia_i^{\sigma})^{q^{n-i}},\]
\[\ldots\]
\[Da_m^{\sigma}-(a_{n-m}A)^{q^m}=\sum_{i=0}^{n-1} (Ba_ia_{i+m}^{\sigma})^{q^{n-i}},\]
\[\ldots\]
\[Da_{n-1}^{\sigma}-(a_1A)^{q^{n-1}}=\sum_{i=0}^{n-1} (Ba_ia_{i+n-1}^{\sigma})^{q^{n-i}},\]
where the indices are taken modulo $n$.
\end{proposition}
\begin{proof}
Because of cardinality reasons the condition $AD-BC\neq 0$ is necessary. Then
\[
\{(x,\hat{f}(x)) \colon x\in \F_{q^n}\}=
\left\{\begin{pmatrix}
A & B \\
C & D \\
\end{pmatrix}
\begin{pmatrix}
x^{\sigma} \\
f(x)^{\sigma} \\
\end{pmatrix}
\colon x\in \F_{q^n} \right\}
\]
holds if and only if
\[Cx^{\sigma}+D\sum_{j=0}^{n-1}a_j^{\sigma} x^{\sigma q^j}=\sum_{i=0}^{n-1}a_{n-i}^{q^i}\left(Ax^{\sigma}+B\sum_{j=0}^{n-1}a_j^{\sigma} x^{\sigma q^j}\right)^{q^i}\]
for each $x\in \F_{q^n}$.
After reducing modulo $x^{q^n}-x$, this is a polynomial equation of degree $q^{n-1}$ in the variable $x^{\sigma}$.
It follows that it holds for each $x\in \F_{q^n}$ if and only if it is the zero polynomial. Comparing coefficients on both sides yields the assertion.
\end{proof}

We are able to prove the following.

\begin{proposition}
\label{nemsimple}
Consider a polynomial of the form $f(x)=\delta x^q+x^{q^{n-1}}$, where $q>4$ is a power of the prime $p$.
If $n>4$, then for each generator $\delta$ of the multiplicative group of $\F_{q^n}$ the linear set $L_f$ is not simple.
\end{proposition}
\begin{proof}
Lemma \ref{lem2} yields $L_f=L_{\hat f}$ thus it is enough to show
the existence of $\delta$ such that there is no $\F_{q^n}$-semilinear map between $U_f$ and $U_{\hat f}$.
In the equations of Proposition \ref{eqn0} we have $a_1=\delta$, $a_{n-1}=1$ and $a_0=a_2=\ldots=a_{n-2}=0$, thus
\[C =(B\delta^{\sigma+1})^{q^{n-1}}+B^q,\]
\[D\delta^{\sigma}-A^q=0,\]
\[0=(B\delta)^{q^{n-1}},\]
\[D-(\delta A)^{q^{n-1}}=0,\]
where $\sigma=p^k$ for some integer $k$.
If there is a solution, then $B=C=0$ and $(\delta A)^{q^{n-1}}\delta^{\sigma}=A^q$.
Taking $q$-th powers on both sides yield
\begin{equation}
\label{eq}
\delta^{\sigma q +1}=A^{q^2-1}
\end{equation}
and hence
\begin{equation}
\label{delta}
\delta^\frac{(\sigma q +1)(q^n-1)}{q-1}=1.
\end{equation}
For each $\sigma$ let $G_\sigma$ be the set of elements $\delta$ of $\F_{q^n}$ satisfying \eqref{delta}.
For each $\sigma$, $G_\sigma$ is a subgroup of the multiplicative group $M$ of $\F_{q^n}$.
We show that these are proper subgroups of $M$.
We have $G_{p^k}=M$ if and only if $q^n-1$ divides $\frac{(p^k q +1)(q^n-1)}{q-1}$, i.e. when
$q-1$ divides $p^k q +1$.
Since $\gcd(p^w+1,p^v-1)$ is always 1,2, or $p^{\gcd(w,v)}+1$,
it  follows that for $q>4$ we cannot have $q-1$ as a divisor of $p^k q+1$.

It follows that for any generator $\delta$ of $M$ we have $\delta \notin \cup_{j} G_{p^j}$
and hence $\delta^{\sigma q +1}\neq A^{q^2-1}$ for each $\sigma$ and for each $A$.
\end{proof}

\begin{remark}
If $q=4$, then \eqref{eq} with $k=2(n-1)+1$ asks for the solution of $\delta^3=A^{15}$. When $5$ does not divide $4^n-1$, then
$\{x^3 \colon x\in \F_{4^n}\}=\{x^{15} \colon x\in \F_{4^n}\}$ and hence for each $\delta$ there exists $A$ such that
$\delta^3=A^{15}$.

If $q=3$, then \eqref{eq} with $k=n-1$ asks for the solution of $\delta^2=A^8$. When $4$ does not divide $3^n-1$, then
$\{x^2 \colon x\in \F_{3^n}\}=\{x^{8} \colon x\in \F_{3^n}\}$ and hence for each $\delta$ there exists $A$ such that
$\delta^2=A^8$.

If $q=2$, then \eqref{eq} with $k=0$ asks for the solution of $\delta^3=A^3$. This equation always has a solution.
\end{remark}

\section{\texorpdfstring{Linear sets of rank 4 of $\mathrm{PG}(1,q^4)$}{Linear sets of rank 4 of PG(1,q4)}}
\label{n=4}

$\F_q$-linear sets of rank two of $\PG(1,q^2)$ are the Baer sublines, which are equivalent.
As we have mentioned in the introduction, subgeometries are simple linear sets, in fact they have $\ZG$-class one (cf. \cite[Theorem 2.6]{LV2013} and \cite[Section 25.5]{JWPH3}).
There are two non-equivalent $\F_q$-linear sets of rank 3 of $\PG(1,q^3)$, the linear sets of size $q^2+q+1$ and those of size $q^2+1$.
Linear sets in both families are equivalent, since the stabilizer of a $q$-order subgeometry $\Sigma$ of $\Sigma^*=\PG(2,q^3)$ is transitive on the set of those points of $\Sigma^*\setminus \Sigma$ which are incident with a line of $\Sigma$ and on the set of points of $\Sigma^*$ not incident with any line of $\Sigma$ (cf. Section \ref{subproj} and \cite{LaVa2010}). 
In the first case we have the linear sets of pseudoregulus type with $\G$-class 1 and $\ZG$-class 2 (cf. Remark \ref{nonsimpleex} and Example \ref{pseudoZG}). In the second case we have the linear sets defined by $\Tr_{q^3/q}$ with $\G$-class and $\ZG$-class 1 (cf. Theorem \ref{THMtrace}, see also \cite[Corollary 6]{FSz2008}).

The main result of this section is that each $\F_q$-linear set of rank 4 of $\PG(1,q^4)$, with maximum field of linearity $\F_q$, is simple (cf. Theorem \ref{n4simple}).

\subsection{Subspaces defining the same linear set}

\begin{lemma}
\label{mainLEMMA2}
Let $f(x)=\sum_{i=0}^{3}a_i x^{q^i}$ and $g(x)=\sum_{i=0}^{3}b_i x^{q^i}$ be two $q$-polynomials over $\F_{q^4}$, such that $L_f=L_g$.
Then
\[\N(a_1)+\N(a_2)+\N(a_3)+a_1^{1+q^2}a_3^{q+q^3}+a_1^{q+q^3}a_3^{1+q^2}+\Tr{q^4/q}\left(a_1a_2^{q+q^2}a_3^{q^3}\right)=\]
\[\N(b_1)+\N(b_2)+\N(b_3)+b_1^{1+q^2}b_3^{q+q^3}+b_1^{q+q^3}b_3^{1+q^2}+\Tr{q^4/q}\left(b_1b_2^{q+q^2}b_3^{q^3}\right).\]
\end{lemma}
\begin{proof}
We are going to follow the proof of Lemma \ref{mainLEMMA}. As in that proof, we may assume $a_0=b_0=0$.
In Lemma \ref{lmain} take $d=1+q+q^2+q^3$.
We obtain
\[\sum_{1\leq i,j,k,m \leq 3} a_ia_j^qa_k^{q^2}a_m^{q^3} \sum_{x\in \F_{q^4}^*} x^{q^i-1+q^{j+1}-q+q^{k+2}-q^2+q^{m+3}-q^3}=\]
\[\sum_{1\leq i,j,k,m \leq 3} b_ib_j^qb_k^{q^2}b_m^{q^3} \sum_{x\in \F_{q^4}^*} x^{q^i-1+q^{j+1}-q+q^{k+2}-q^2+q^{m+3}-q^3}.\]
$\sum_{x\in \F_{q^4}^*} x^{q^i-1+q^{j+1}-q+q^{k+2}-q^2+q^{m+3}-q^3}=-1$ if and only if
\[q^i+q^{j+1}+q^{k+2}+q^{m+3}\equiv q^i+q^{j+1}+q^{k+2}+q^{m-1} \equiv 1+q+q^2+q^3 \pmod {q^4-1},\]
and zero otherwise. Suppose that the former case holds.

First consider $k=1$.
Then $q^i+q^{j+1}+q^{k+2}+q^{m-1} \leq q^3+q^4+q^3+q^2<1+q+q^2+q^3+2(q^4-1)$ hence one of the following holds.
    \begin{itemize}
        \item If $q^i+q^{j+1}+q^{k+2}+q^{m-1}= 1+q+q^2+q^3$, then $m=i=j=k=1$.
        \item If $q^i+q^{j+1}+q^{k+2}+q^{m-1}= 1+q+q^2+q^3+q^4-1=q+q^2+q^3+q^4$, then $\{i,j+1,k+2,m-1\}=\{1,2,3,4\}$, hence one of the following holds
            \[i=1,\, j=3,\, k=1,\,m=3 ,\]
            \[i=2,\, j=3,\, k=1,\,m=2 .\]
     \end{itemize}

Now consider the case $k \geq 2$.
Then $q^i+q^{j+1}+q^{k+2}+q^{m-1} \equiv q^i+q^{j+1}+q^{k-2}+q^{m-1}\leq q^3+q^4+q+q^2<1+q+q^2+q^3+2(q^4-1)$ hence one of the following holds.
    \begin{itemize}
        \item If $q^i+q^{j+1}+q^{k-2}+q^{m-1}= 1+q+q^2+q^3$, then $\{i,j+1,k-2,m-1\}=\{0,1,2,3\}$, hence one of the following holds
            \[i=1,\, j=2,\, k=2,\,m=3 ,\]
            \[i=2,\, j=2,\, k=2,\,m=2 ,\]
            \[i=2,\, j=2,\, k=3,\,m=1 ,\]
            \[i=3,\, j=1,\, k=2,\,m=2 ,\]
            \[i=3,\, j=1,\, k=3,\,m=1 .\]
        \item If $q^i+q^{j+1}+q^{k-2}+q^{m-1}= 1+q+q^2+q^3+q^4-1=q+q^2+q^3+q^4$, then $i=j=k=m=3$.
     \end{itemize}
\end{proof}

\begin{proposition}
\label{prop:n4}
Let $f(x)$ and $g(x)$ be two $q$-polynomials over $\F_{q^4}$ such that $L_f=L_g$.
If the maximum field of linearity of $f$ is $\F_q$, then
\[g(x)=f(\lambda x)/\lambda,\]
or
\[g(x)=\hat{f}(\lambda x)/\lambda.\]
\end{proposition}
\begin{proof}
By Proposition \ref{rankisgood}, the maximum field of linearity of $g$ is also $\F_q$.
First note that $L_g=L_f$ when $g$ is as in the assertion (cf. Lemmas \ref{lem2} and \ref{lem3}).
Let $f(x)=\sum_{i=0}^3a_ix^{q^i}$ and $g(x)=\sum_{i=0}^3b_ix^{q^i}$.

First we are going to use Lemma \ref{mainLEMMA}. From \eqref{a0ba0} we have $a_0=b_0$. From \eqref{e1} with $n=4$ and $k=1,2$ we have $a_1a_3^q=b_1b_3^q$ and $a_2^{1+q^2}=b_2^{1+q^2}$, respectively. From \eqref{e2} with $n=4$ and $k=2$ we obtain
\begin{equation}
\label{long}
a_1^{q+1}a_2^{q^2}+a_2a_3^{q+q^2}=b_1^{q+1}b_2^{q^2}+b_2b_3^{q+q^2}.
\end{equation}

Note that $a_1a_3^q=b_1b_3^q$ implies
\begin{equation}
\label{prodnorms}
\N(b_1)\N(b_3)=\N(a_1)\N(a_3).
\end{equation}

Multiplying \eqref{long} by $b_2$ and applying $a_2^{1+q^2}=b_2^{1+q^2}$ yields:
\begin{equation}
\label{long2}
b_2^2b_3^{q^2+q}-b_2(a_1^{q+1}a_2^{q^2}+a_2a_3^{q^2+q})+b_1^{q+1}a_2^{q^2+1}=0.
\end{equation}

First suppose $b_1b_2b_3\neq 0$. Then \eqref{long2} is a second degree polynomial in $b_2$.
Applying $a_1a_3^q=b_1b_3^q$ it is easy to see that the roots of \eqref{long2} are
\[b_{2,1}=\frac{a_1^{q+1}a_2^{q^2}}{b_3^{q^2+q}},\]
\[b_{2,2}=\frac{a_2a_3^{q^2+q}}{b_3^{q^2+q}}.\]

First we consider $b_2=b_{2,1}$. Then $a_2^{1+q^2}=b_2^{1+q^2}$ yields $\N(a_1)=\N(b_3)$ and hence $\N(b_1)=\N(a_3)$.
In particular, $\N(b_1/a_3^q)=1$ and hence $b_1=a_3^q\lambda^{q-1}$ for some $\lambda\in \F_{q^4}^*$.
From $a_1a_3^q=b_1b_3^q$ we obtain $b_3=a_1^{q^3}a_3/b_1^{q^3}=a_1^{q^3}\lambda^{q^3-1}$.
Applying this we get $b_2=a_1^{q+1}a_2^{q^2}/b_3^{q^2+q}=a_2^{q^2}\lambda^{q^2-1}$ and hence
\[g(x)=a_0x+a_3^q\lambda^{q-1}x^q+a_2^{q^2}\lambda^{q^2-1}x^{q^2}+a_1^{q^3}\lambda^{q^3-1}x^{q^3}.\]
as we claimed.

Now consider $b_2=b_{2,2}$. Then $a_2^{1+q^2}=b_2^{1+q^2}$ yields $\N(a_3)=\N(b_3)$ and hence $\N(a_1)=\N(b_1)$.
Hence $b_1=a_1\lambda^{q-1}$ for some $\lambda\in \F_{q^4}^*$.
From $a_1a_3^q=b_1b_3^q$ we obtain $b_3=a_1^{q^3}a_3/b_1^{q^3}=a_3\lambda^{q^3-1}$.
Applying this we obtain $b_2=a_2a_3^{q^2+q}/b_3^{q^2+q}=a_2\lambda^{q^2-1}$ and hence
\[g(x)=a_0x+a_1 \lambda^{q-1}x^q+a_2^{q^2}\lambda^{q^2-1}x^{q^2}+a_3^{q^3}\lambda^{q^3-1}x^{q^3}.\]

If $b_1=b_3=0$, then either $b_2=0$ and the maximum field of linearity of $g(x)$ is $\F_{q^4}$, or $b_2\neq 0$ and the maximum field of linearity of $g(x)$ is $\F_{q^2}$. Thus we may assume $b_1\neq 0$ or $b_3\neq 0$.

First assume $b_2\neq 0$ and $b_1=0$. Then $b_3\neq 0$ and \eqref{long2} gives
\[b_2b_3^{q^2+q}=a_1^{q+1}a_2^{q^2}+a_2a_3^{q^2+q}.\]

Then $a_1a_3^q=b_1b_3^q$ yields either $a_1=0$ and $b_2b_3^{q^2+q}=a_2a_3^{q^2+q}$, or $a_3=0$ and $b_2b_3^{q^2+q}=a_1^{q+1}a_2^{q^2}$. Taking $(q^2+1)$-powers on both sides gives $b_2^{q^2+1}\N(b_3)=a_2^{q^2+1}\N(a_3)$, or $b_2^{q^2+1}\N(b_3)=\N(a_1)a_2^{q^2+1}$, respectively. Applying $b_2^{q^2+1}=a_2^{q^2+1}$ we get
$\N(b_3)=\N(a_3)$, or $\N(b_3)=\N(a_1)$, respectively. Note that the set of elements with norm 1 in $\F_{q^4}$ is $\{x^{q^3-1} \colon x\in \F_{q^4}^*\}$, thus
in the first case there exists $\lambda \in \F_{q^4}^*$ such that $b_3=a_3 \lambda^{q^3-1}$. Then $b_2b_3^{q^2+q}=a_2a_3^{q^2+q}$ yields $b_2=a_2\lambda^{q^2-1}$ and hence $g(x)=a_0x+a_2\lambda^{q^2-1}x^{q^2}+a_3\lambda^{q^3-1}x^{q^3}$.
In the second case the same reasoning yields $g(x)=a_0x+a_2^{q^2}\lambda^{q^2-1}x^{q^2}+a_1^{q^3}\lambda^{q^3-1}x^{q^3}$.

If $b_2\neq 0$ and $b_3=0$, then the coefficient of $x^q$ in $\hat{g}(x)$ is zero and the assertion follows from the above arguments applied to $\hat{g}$ instead of $g$.

Now assume $b_2=0$ and $b_1b_3=0$. Then $L_g=L_f$ is a linear set of pseudoregulus type and hence the assertion also follows from \cite{LaShZa2013}.
For the sake of completeness we present a proof also in this case. Equation $b_2^{q^2+1}=a_2^{q^2+1}$ yields $a_2=0$ and equation $a_1a_3^q=b_1b_3^q$ yields $a_1a_3=0$. Then from Lemma \ref{mainLEMMA2} we have
\begin{equation}
\label{sumnorms}
\N(a_1)+\N(a_3)=\N(b_1)+\N(b_3).
\end{equation}
If $b_1=0$, then $b_3\neq 0$ and either $a_1=0$ and $\N(a_3)=\N(b_3)$, or $a_3=0$ and $\N(a_1)=\N(b_3)$.
In the first case $g(x)=a_0x+a_3\lambda^{q^3-1}x^{q^3}$, in the second case $g(x)=a_0x+a_1^q \lambda^{q^3-1}x^{q^3}$.
If $b_3=0$, then $b_1\neq 0$ and either $a_1=0$ and $\N(a_3)=\N(b_1)$, or $a_3=0$ and $\N(a_1)=\N(b_1)$.
In the first case $g(x)=a_0x+a_3^q\lambda^{q-1}x^q$, in the second case $g(x)=a_0x+a_1\lambda^{q-1}x^q$.

There is only one case left, when $b_2=0$ and $b_1b_3\neq 0$.
Then from Lemma \ref{mainLEMMA2} and from $a_1a_3^q=b_1b_3^q$ it follows that
\begin{equation}
\label{sumnorms2}
\N(a_1)+\N(a_3)
=\N(b_1)+\N(b_3).
\end{equation}
Together with \eqref{prodnorms} it follows that either $\N(a_1)=\N(b_1)$ and $\N(a_3)=\N(b_3)$, or
$\N(a_1)=\N(b_3)$ and $\N(a_3)=\N(b_1)$. In the first case
$g(x)=a_0x+a_1\lambda^{q-1}x^q+a_3\lambda^{q^3-1}x^{q^3}$,
in the second case
$g(x)=a_0x+a_3^q\lambda^{q-1}x^q+a_1^{q^3}\lambda^{q^3-1}x^{q^3}$, for some $\lambda\in \F_{q^4}^*$.
\end{proof}

Now we are able to prove the following.

\begin{theorem}
Let $L_{U}$ be an $\F_q$--linear set of a line $\PG(W,\F_{q^4})$ of rank $4$, with maximum field of linearity $\F_q$, and let $\beta$ be a non--degenerate alternating form of $W$. If $V$ is an $\F_q$--vector subspace of $W$ such that $L_U=L_V$, then either
\[V=\mu U,\]
or
\[V=\mu U^{\perp_\beta'},\]
for some $\mu\in\F_{q^4}^*$, where $\perp_\beta'$ is the orthogonal complement map induced by $\Tr_{q^4/q}\circ \beta$ on the lattice of the $\F_q$--subspaces of $W$.
\end{theorem}
\begin{proof}
First of all, observe that if $\beta_1$ is another non--degenerate alternating form of $W$ and $\perp_{\beta_1}'$ is the corresponding orthogonal complement
map induced on the lattice the $\F_{q}$-subspaces of $W$, direct computations show that there exists $a\in\F_{q^n}^*$ such that $\beta_1=a\beta$ and for each $\F_q$--vector subspace $S$ of $W$ we get $S^{\perp_{\beta}'}=a S^{\perp_{\beta_1}'}$.

Let $\phi$ be the collineation of $\PG(W,\F_{q^4})$ such that $L_U^\phi$ does not contain the point $\langle (0,1)\rangle_{\F_{q^4}}$. Then $L_{U^{\varphi}}=L_{V^{\varphi}}$, where $\varphi$ is the invertible $\F_{q^4}$-semilinear map of $W$ inducing $\phi$, and $\sigma$ is the associated field automorphism. Also, $U^{\varphi}=U_f$ and $V^{\varphi}=V_g$ for two $q$--polynomials $f$ and $g$ over $\F_{q^4}$. Since $L_f=L_g$, by Proposition \ref{prop:n4} and by Lemma \ref{lem3}, taking also (\ref{form:hatf}) into account, it follows that there exists $\lambda\in\F_{q^4}^*$ such that either $\lambda V_g=U_f$ or $\lambda V_g=U_{\hat f}=U_f^{\perp'}$, where $\perp'$ is the orthogonal complement map induced by the non-degenerate alternating form defined in (\ref{form:perp1}). In the first case we have that $V=\mu U$, where $\mu=\frac 1{\lambda^{\sigma^{-1}}}$. In the second case we have $V=\frac 1{\lambda^{\sigma^{-1}}} U^{\varphi\,\perp'\,\varphi^{-1}}$. The map ${\varphi\perp'\varphi^{-1}}$ defines the orthogonal complement map on the lattice the $\F_{q}$-subspaces of $W$ induced by another non--degenerate alternating form of $W$. As observed above, there exists $a\in\F_{q^4}^*$ such that $U^{\varphi\,\perp'\,\varphi^{-1}}=a U^{\perp_\beta'}$. The assertion follows with $\mu=\frac a{\lambda^{\sigma^{-1}}}$.
\end{proof}

\subsection{\texorpdfstring{Semilinear maps between $U_f$ and $U_{\hat{f}}$}{Semilinear maps between Uf and Uf'}}

The next result is just Proposition \ref{eqn0} with $n=4$.

\begin{corollary}
\label{eqn}
Let $f(x)=a_0x+a_1x^q+a_2x^{q^2}+a_3x^{q^3}$. There is an $\F_{q^4}$-semilinear
map between $U_f$ and $U_{\hat{f}}$ if and only if the following system of four equations has a solution $A,B,C,D \in \F_{q^4}$, $AD-BC\neq 0$, $\sigma = p^k$.
\[C + Da_0^{\sigma}-a_0A=Ba_0a_0^{\sigma}+(Ba_1a_1^{\sigma})^{q^3}+(Ba_2a_2^{\sigma})^{q^2}+(Ba_3a_3^{\sigma})^q,\]
\[Da_1^{\sigma}-(a_3A)^q=Ba_0a_1^{\sigma}+(Ba_1a_2^{\sigma})^{q^3}+(Ba_2a_3^{\sigma})^{q^2}+(Ba_3a_0^{\sigma})^q,\]
\[Da_2^{\sigma}-(a_2A)^{q^2}=Ba_0a_2^{\sigma}+(Ba_1a_3^{\sigma})^{q^3}+(Ba_2a_0^{\sigma})^{q^2}+(Ba_3a_1^{\sigma})^q,\]
\[Da_3^{\sigma}-(a_1A)^{q^3}=Ba_0a_3^{\sigma}+(Ba_1a_0^{\sigma})^{q^3}+(Ba_2a_1^{\sigma})^{q^2}+(Ba_3a_2^{\sigma})^q.\]
\end{corollary}

\begin{theorem}
\label{n4simple}
Linear sets of rank 4 of $\PG(1,q^4)$, with maximum field of linearity $\F_q$, are simple.
\end{theorem}
\begin{proof}
Let $f=\sum_{i=0}^3a_ix^{q^i}$.
After a suitable projectivity we may assume $a_0=0$.
We will use Corollary \ref{eqn} with $\sigma\in \{1,q^2\}$.
We may assume that $a_1=0$ and $a_3=0$ do not hold at the same time since otherwise $f$ is $\F_{q^2}$-linear.

First consider the case when $\N(a_1)=\N(a_3)$.
Let $B=C=0$, $D=A^{q^2}$ and take $A$ such that $A^{q-1}=a_3/a_1^q$. This can be done since $\N(a_3/a_1^q)=1$. Then
Corollary \ref{eqn} with $\sigma=q^2$ provides the existence of an $\F_{q^4}$-semilinear map between $U_f$ and $U_{\hat{f}}$.

From now on we assume $\N(a_1)\neq \N(a_3)$.

If $a_2=a_1=0$, then let $\sigma=1$, $A=D=0$, $B=1$ and $C=a_3^{2q}$.
If $a_2=a_3=0$, then let $\sigma=1$, $A=D=0$, $B=1$ and $C=a_1^{2q^3}$.

Now consider the case $a_2=0$ and $a_1a_3\neq 0$.
Let $A=D=0$. Then the equations of Corollary \ref{eqn} with $\sigma=1$ yield
\begin{equation}
\label{ff}
C=B^{q^3}a_1^{2q^3}+B^qa_3^{2q},
\end{equation}
\begin{equation}
\label{tt}
0=B^qa_1^qa_3^q+B^{q^3}a_1^{q^3}a_3^{q^3}.
\end{equation}
\eqref{tt} is equivalent to $0=(Ba_1a_3)^{q^2}+Ba_1a_3$.
Since $X^{q^2}+X=0$ has $q^2$ solutions in $\F_{q^4}$, for any $a_1$ and $a_3$ we can find $B\in\F_{q^4}^*$ such that \eqref{tt} is satisfied. If $B^{q^3}a_1^{2q^3}+B^qa_3^{2q}\neq 0$, then let $C$ be this field element. We show that this is always the case.
Suppose, contrary to our claim, that $B^{q^3-q}=-a_3^{2q}/a_1^{2q^3}$. Because of the choice of $B$ \eqref{tt} yields
$B^{q^3-q}=-a_1^{q-q^3}a_3^{q-q^3}$. Since $B\neq 0$ this implies
\[-a_3^{2q}/a_1^{2q^3}=-a_1^{q-q^3}a_3^{q-q^3},\]
and hence $a_1^{q^2+1}=a_3^{q^2+1}$. A contradiction since $\N(a_1)\neq \N(a_3)$.
From now on we assume $a_2\neq 0$, we may also assume $a_2=1$ after a suitable projectivity.

Corollary \ref{eqn} with $\sigma=1$ yields
\begin{equation}
\label{p1}
C =(Ba_1^2)^{q^3}+B^{q^2}+(Ba_3^2)^q,
\end{equation}
\begin{equation}
\label{p2}
Da_1-(a_3A)^q=(Ba_1)^{q^3}+(Ba_3)^{q^2},
\end{equation}
\begin{equation}
\label{p3}
D-A^{q^2}=(Ba_1a_3)^{q^3}+(Ba_3a_1)^q,
\end{equation}
\begin{equation}
\label{p4}
Da_3-(a_1A)^{q^3}=(Ba_1)^{q^2}+(Ba_3)^q.
\end{equation}

The right hand side of \eqref{p2} is the $q$-th power of the right hand side of \eqref{p4} and hence $D^qa_3^q-a_1A=Da_1-a_3^qA^q$, i.e.
\[a_3^q(D+A)^q=a_1(D+A).\]
Since $a_1$ or $a_3$ is non-zero, we have either $D=-A$, or $(D+A)^{q-1}=a_1/a_3^q$. The latter case can be excluded since in that case $\N(a_1)=\N(a_3)$. Let $D=-A$. Then the left hand side of \eqref{p2} is $w(A):=-Aa_1-a_3^qA^q$.
The kernel of $w$ is trivial and hence $B$ uniquely determines $A$.
The inverse of $w$ is
\[w^{-1}(x)=\frac{-xa_1^{q+q^2+q^3}+x^qa_1^{q^2+q^3}a_3^q-x^{q^2}a_1^{q^3}a_3^{q+q^2}+x^{q^3}a_3^{q+q^2+q^3}}{\N(a_1)-\N(a_3)}.\]
Denote the right hand side of \eqref{p2} by $r(B)$, the right hand side of \eqref{p3} by $t(B)$.
Then $B$ has to be in the kernel of
\[K(x):=w^{-1}(r(x))+(w^{-1}(r(x)))^{q^2}+t(x).\]
If $B=0$, then $A=B=D=0$ and hence this is not a suitable solution.
It is easy to see that $Im\, t \subseteq \F_{q^2}$ and hence also
$Im\, K \subseteq \F_{q^2}$, so the kernel of $K$ has at least dimension 2.

Let $B\in \ker K$, $B\neq 0$, $A:=w^{-1}(r(B))$ and
$C:=(Ba_1^2)^{q^3}+B^{q^2}+(Ba_3^2)^q$ (we recall $D=-A$). This gives a solution.
We have to check that $B$ can be chosen such that $AD-BC\neq 0$, i.e.
\[Q(B):=\left(w^{-1}(r(B))\right)^2+ B\left((Ba_1^2)^{q^3}+B^{q^2}+(Ba_3^2)^q\right),\]
is non-zero. We have $w^{-1}(r(x))(\N(a_1)-\N(a_3))=\sum_{i=0}^3c_i x^{q^i}$, where
\[c_0=a_1^{1+q^2+q^3}a_3^q-a_1^{q^3}a_3^{1+q+q^2},\]
\[c_1=a_3^{2q+q^2+q^3}-a_1^{q+q^3}a_3^{q+q^2},\]
\[c_2=a_3^{q+q^2+q^3}a_1^{q^2}-a_1^{q+q^2+q^3}a_3^{q^2},\]
\[c_3=a_1^{q^2+q^3}a_3^{q+q^3}-a_1^{q+q^2+2q^3}.\]
If $X_0, X_1, X_2, X_3$ denote the coordinate functions in $\PG(3,q^4)$ and $Q(B)=0$ for some $B\in \F_{q^4}$, then the point $\la(B,B^q,B^{q^2},B^{q^3})\ra_{q^4}$ is contained in the the quadric $\cQ$ of $\PG(3,q^4)$ defined by the equation
\[\left(\sum_{i=0}^3c_i X_i\right)^2+X_0(X_1a_3^{2q}+X_2+X_3a_1^{2q^3})(\N(a_1)-\N(a_3))^2=0.\]
We can see that the equation of $\cQ$ is the linear combination of the equations of two degenerate quadrics, a quadric of rank 1 and a quadric of rank 2.
It follows that $\cQ$ is always singular and it has rank 2 or 3.
In particular, the rank of $\cQ$ is 2 when the intersection of the planes $\cA: X_0=0$ and $\cB: X_1a_3^{2q}+X_2+X_3a_1^{2q^3}=0$ is contained in the plane $\cC: \sum_{i=0}^3c_i X_0=0$. Straightforward calculations show that under our hypothesis ($a_1\neq 0$ or $a_3\neq 0$, $\N(a_1)\neq \N(a_3)$) this happens if only if $1=a_1^qa_3$.

\noindent
We recall that the kernel of $K$ has dimension at least two. Let
\[H=\{\la(x,x^q,x^{q^2},x^{q^3})\ra_{q^4} \colon K(x)=0\}.\]

Our aim is to prove that $H$ has points not belonging to the quadric $\cQ$, i.e. $H \nsubseteq \cQ$.
\medskip

\noindent
Note that $x\in \F_{q^4} \mapsto (x,x^q,x^{q^2},x^{q^3})\in \F_{q^4}^4$ is a vector-space isomorphism between $\F_{q^4}$ and the 4-dimensional $\F_q$-space $\{(x,x^q,x^{q^2},x^{q^3}) \colon x\in \F_{q^4}\} \subset \F_{q^4}^4$.
Denote by $\bar{H}$ the $\F_{q^4}$-extension of $H$, i.e. the projective subspace of $\PG(3,q^4)$ generated by the points of $H$. Then the projective dimension of $\bar{H}$ is $\dim \ker K -1$. Let $\sigma$ denotes the collineation $(X_0,X_1,X_2,X_3)\mapsto (X_3^q,X_0^q,X_1^q,X_2^q)$ of $\PG(3,q^4)$. Then the points of $H$ are fixed points of $\sigma$ and hence $\sigma$ fixes the subspace $\bar{H}$.
Note that the vertex of $\cQ$ is always disjoint from $H$ since it is contained in $\cA$, while $H$ is disjoint from it.

First of all note that if $\dim \ker K=4$, i.e. $K$ is the zero polynomial, then $H$ is a subgeometry of $\PG(3,q^4)$ isomorphic to $\PG(3,q)$, which clearly cannot be contained in $\cQ$. It follows that $\dim \ker K$ is either 3 or 2, i.e. $H$ is either a $q$-order subplane or a $q$-order subline.

First assume $1\neq a_1^qa_3$, i.e. the case when $\cQ$ has rank 3.
If $H$ is a $q$-order subplane, then $H$ cannot be contained in $\cQ$. To see this, suppose the contrary and take three non-concurrent $q$-order sublines of $H$. The $\F_{q^4}$-extensions of these sublines are also contained in $\cQ$, but there is at least one of them which does not pass through the singular point of $\cQ$, a contradiction.
Now assume that $H$ is a $q$-order subline.
The singular point of $\cQ$ is the intersection of the planes
$\cA, \cB$ and $\cC$. Straightforward calculations show that this point is
$V=\la (v_0,v_1,v_2,v_3)\ra_{q^4}$, where
\[v_0=0,\]
\[v_1=a_1^{q^2+q^3}(a_1^{q^3}a_3^{q^2}-1),\]
\[v_2=a_1^{q^3}a_3^{q}(a_1^{q^2}a_3^q-a_1^{q^3}a_3^{q^2}),\]
\[v_3=a_3^{q+q^2}(1-a_1^{q^2}a_3^q).\]

Suppose, contrary to our claim, that $H$ is contained in $\cQ$. Then $\bar{H}$ passes through the singular point $V$ of $\cQ$. Since $\bar{H}$ is fixed by $\sigma$, it follows that the points $V,V^{\sigma},V^{\sigma^2},V^{\sigma^3}$ have to be collinear ($v_0=0$ yields that these four points cannot coincide). Let $M$ denote the $4\times 4$ matrix, whose $i$-th row consists of the coordinates of $V^{\sigma^{i-1}}$ for $i=1,2,3,4$.
The rank of $M$ is two, thus each of its minors of order three is zero. Let $M_{i,j}$ denote the submatrix of $M$ obtained by deleting the $i$-th row and $j$-th column of $M$. Then
\[\det M_{1,2}=a_1^{q+1}(a_1^qa_3-1)^{q^3+1}\alpha,\]
\[\det M_{1,4}=a_3^{q^3+1}(a_1^qa_3-1)^{q^3+1}\beta,\]
where
\[\alpha=\N(a_1)(a_1^{q^2}a_3^q-1)+\N(a_3)(1-a_1^qa_3-a_1^{q^3}a_3^{q^2}+a_1a_3^{q^3}),\]
\[\beta=\N(a_1)(a_1a_3^{q^3}+a_1^{q^2}a_3^q-a_1^qa_3-1)+\N(a_3)(1-a_1^{q^3}a_3^{q^2}).\]
Since $a_1$ and $a_3$ cannot be both zeros and $a_1^qa_3-1\neq 0$, we have $\alpha=\beta=0$.
But $\alpha-\beta=(\N(a_1)-\N(a_3))(a_1^qa_3-a_1a_3^{q^3})$. It follows that $a_1^qa_3 \in \F_q$ and hence $\alpha$ can be written as $(\N(a_1)-\N(a_3))(a_1^qa_3-1)$,  which is non-zero. This contradiction shows that $V$ cannot be contained in a line fixed by $\sigma$ and hence $\bar{H}$ cannot pass through $V$. It follows that $H \nsubseteq\cQ$ and hence we can choose $B$ such that $AD-BC\neq 0$.

Now consider the case $1=a_1^qa_3$. Then $\cQ$ is the union of two planes meeting each other in $\ell:=\cA \cap \cB$.
It is easy to see that $R:=\la(0,1,-a_3^{2q},0)\ra_{q^4}$ and $R^{\sigma}$ are two distinct points of $\ell$.
Since $\N(a_1)\neq \N(a_3)$ and $\N(a_1)\N(a_3)=1$, $\det\{R,R^{\sigma},R^{\sigma^2},R^{\sigma^3}\}=\N(a_3)^2-1$ cannot be zero and hence $R\notin H$, otherwise $\dim  \langle R,R^\sigma,R^{\sigma^2},R^{\sigma^3}\rangle\leq \dim \bar H\leq 2$.
Suppose, contrary to our claim, that $H$ is contained in one of the two planes of $\cQ$.
Since $R\notin H$, such a plane can be written as $\la H, R \ra$ and since
$H$ is fixed by $\sigma$ and $\ell \subseteq \langle H,R\rangle$, we have $\la H, R \ra^\sigma=\la H, R^\sigma \ra=\la H, R \ra$. Thus
$R,R^\sigma,R^{\sigma^2},R^{\sigma^3}$ are coplanar, a contradiction.
\end{proof}

\section{Different aspects of the classes of a linear set}
\label{aspects}

\subsection{Class of a linear set and the associated variety}

Let $L_U$ be an $\F_q$-linear set of rank $k$ of $\PG(W,\F_{q^n})=\PG(r-1,q^n)$.
Consider the projective space $\Omega=\PG(W,\F_q)=\PG(rn-1,q)$. For each point $P=\la {\bf u} \ra_{\F_{q^n}}$ of $\PG(W,\F_{q^n})$ there corresponds a projective $(n-1)$-subspace $X_P:=\PG(\la {\bf u} \ra_{q^n}, \F_q )$ of $\Omega$. The variety of $\Omega$ associated to $L_U$ is
\begin{equation}
\cV_{r,n,k}(L_U)=\bigcup_{P\in L_U} X_P.
\end{equation}
A $(k-1)$-space $\cH=\PG(V,\F_q)$ of $\Omega$ is said to be a \emph{transversal} space of $\cV(L_U)$ if $\cH \cap X_P \neq \emptyset$ for each point $P\in L_U$, i.e. $L_U=L_V$.

The $\ZG$-class of an $\F_q$-linear set $L_U$ of rank $n$ of $\PG(W,\F_{q^n})=\PG(1,q^n)$, with maximum field of linearity $\F_q$, is the number of transversal spaces of $\cV_{2,n,n}(L_U)$ up to the action of the subgroup $G$ of $\mathrm{PGL}(2n-1,q)$ induced by the maps ${\bf x}\in W \mapsto \lambda {\bf x} \in W$, with $\lambda \in \F_{q^n}^*$. Note that $G$ fixes $X_P$ for each point $P\in \PG(1,q^n)$ and hence fixes the variety.

The maximum size of an $\F_q$-linear set $L_U$ of rank $n$ of $\PG(1,q^n)$ is $(q^n-1)/(q-1)$. If this bound is attained (hence each point of $L_U$ has weight one), then $L_U$ is a \emph{maximum scattered} linear set of $\PG(1,q^n)$. For maximum scattered linear sets, the number of transversal spaces through $Q\in \cV(L_U)$ does not depend on the choice of $Q$ and this number is the $\ZG$-class of $L_U$.

\begin{example}
\label{pseudoZG}
Let $U=\{(x,x^q) \colon x\in \F_{q^n}\}$ and consider the linear set $L_U$.
In \cite{LaShZa2013} the variety $\cV_{2,n,n}(L_U)$ was studied, and the transversal spaces were determined.
It follows that the $\ZG$-class of $L_U$ is $\varphi(n)$, where $\varphi$ is the Euler's phi function.
\end{example}


\subsection{Classes of linear sets as projections of subgeometries}
\label{subproj}

Let $\Sigma=\PG(k-1,q)$ be a canonical subgeometry of $\Sigma^*=\PG(k-1,q^n)$. Let $\Gamma \subset \Sigma^* \setminus \Sigma$  be a $(k-r-1)$-space and let $\Lambda \subset \Sigma^* \setminus \Gamma$ be an $(r-1)$-space of $\Sigma^*$. The projection of $\Sigma$ from {\it center} $\Gamma$ to {\it axis} $\Lambda$ is the point set
\begin{equation}
\label{proj}
L=p_{\,\Gamma,\,\Lambda}(\Sigma):=\{\la \Gamma, P \ra \cap \Lambda \colon P\in \Sigma\}.
\end{equation}

In \cite{LuPo2004} Lunardon and Polverino characterized linear sets as projections of canonical subgeometries. They proved the following.

\begin{theorem}[{\cite[Theorems 1 and 2]{LuPo2004}}]
\label{LuPo}
Let $\Sigma^*$, $\Sigma$, $\Lambda$, $\Gamma$ and $L=p_{\,\Gamma,\,\Lambda}(\Sigma)$ be defined as above. Then $L$ is an $\F_q$-linear set of rank $k$ and $\la L \ra=\Lambda$. Conversely, if $L$ is an $\F_q$-linear set of rank $k$ of $\Lambda=\PG(r-1,q^n)\subset \Sigma^*$ and $\la L \ra=\Lambda$, then there is a $(k-r-1)$-space $\Gamma$ disjoint from $\Lambda$ and a canonical subgeometry $\Sigma=\PG(r-1,q)$ disjoint from  $\Gamma$ such that $L=p_{\,\Gamma,\,\Lambda}(\Sigma)$.
\end{theorem}

Let $L_U$ be an $\F_q$-linear set of rank $k$ of $\mathbb{P}=\PG(W,\F_{q^n})=\PG(r-1,q^n)$ such that for each $k$-dimensional $\F_q$-subspace $V$ of $W$
if $\PG(V,\F_q)$ is a transversal space of $\cV_{r,n,k}(L_U)$, then there exists
$\gamma \in \mathrm{P \Gamma L}(W,\F_q)$, such that
$\gamma$ fixes the Desarguesian spread $\{X_P \colon P\in \mathbb{P}\}$ and $\PG(U,\F_q)^{\gamma}=\PG(V,\F_q)$. This is condition (A) from \cite{CSZ2015}, and it is equivalent to say that $L_U$ is a simple linear set. Then the main results of \cite{CSZ2015} can be formalized as follows.

\begin{theorem}[\cite{CSZ2015}]
Let $L_1=p_{\,\Gamma_1,\,\Lambda_1}(\Sigma_1)$ and $L_2=p_{\,\Gamma_2,\,\Lambda_2}(\Sigma_2)$ be two linear sets of rank $k$.
If $L_1$ and $L_2$ are equivalent and one of them is simple, then there is a collineation mapping $\Gamma_1$ to $\Gamma_2$ and $\Sigma_1$ to $\Sigma_2$.
\end{theorem}

\begin{theorem}[\cite{CSZ2015}]
If $L$ is a non-simple linear set of rank $k$ in $\Lambda=\la L \ra$, then there are a subspace  $\Gamma=\Gamma_1=\Gamma_2$ disjoint from $\Lambda$, and two $q$-order canonical subgeometries $\Sigma_1,\Sigma_2$ such that $L=p_{\,\Gamma,\,\Lambda}(\Sigma_1)=p_{\,\Gamma,\,\Lambda}(\Sigma_2)$, and there is no collineation fixing $\Gamma$ and mapping $\Sigma_1$ to $\Sigma_2$.
\end{theorem}

Now we interpret the classes of linear sets, hence we are going to consider $\F_q$-linear sets of rank $n$ of $\Lambda=\PG(1,q^n)=\PG(W,\F_{q^n})$, with maximum field of linearity $\F_q$. Arguing as in the proof of \cite[Theorem 7]{CSZ2015}, if $L_U$ is non-simple, then for any pair $U$, $V$ of $n$-dimensional $\F_q$-subspaces of $W$ with $L_U=L_V$ such that $U^f \neq V$ for each $f\in \Gamma \mathrm{L}(2,q^n)$ we can find a $q$-order subgeometry $\Sigma$ of $\Sigma^*=\PG(n-1,q^n)$ and two $(n-3)$-spaces $\Gamma_1$ and $\Gamma_2$ of $\Sigma^*$, disjoint from $\Sigma$ and from $\Lambda$, lying on different orbits of $Stab(\Sigma)$.
On the other hand, arguing as in \cite[Theorem 6]{CSZ2015}, if there exist two $(n-3)$-subspaces $\Gamma_1$ and $\Gamma_2$ of $\Sigma^*$, disjoint from $\Sigma$ and from $\Lambda$, belonging to different orbits of $Stab(\Sigma)$ and such that $L=p_{\Lambda,\, \Gamma_1}(\Sigma)=p_{\Lambda,\, \Gamma_2}(\Sigma)$, then it is possible to construct two $n$-dimensional $\F_q$-subspaces $U$ and $V$ of $W$ with $L_U=L_V$ such that $U^f \neq V$ for each $f\in \Gamma \mathrm{L}(2,q^n)$. Hence we can state the following.
\medskip

The $\Gamma \mathrm{L}$-class of $L_U$ is the number of orbits of $Stab(\Sigma)$ on $(n-3)$-spaces of $\Sigma^*$ containing a $\Gamma$ disjoint from $\Sigma$ and from $\Lambda$ such that $p_{\Lambda,\, \Gamma}(\Sigma)$ is equivalent to $L_U$.

\subsection{Class of linear sets and linear blocking sets of R\'edei type}

A \emph{blocking set} $\cB$ of $\PG(V,\F_{q^n})=\PG(2,q^n)$ is a point set meeting every line of the plane.
Blocking sets of size $q^n+N\leq 2q^n$ with an $N$-secant are called blocking sets of \emph{R\'edei type}, the $N$-secants of the blocking set are called \emph{R\'edei lines}. Let $L_U$ be an $\F_q$-linear set of rank $n$ of a line $\ell=\PG(W,\F_{q^n})$, $W\leq V$, and let ${\bf w}\in V \setminus W$.
Then $\la U, {\bf w} \ra_{\F_q}$ defines an $\F_q$-linear blocking set of $\PG(2,q^n)$ with R\'edei line $\ell$.
The following theorem tells us the number of inequivalent blocking sets obtained in this way.

\begin{theorem}
The $\G$-class of an $\F_q$-linear set $L_U$ of rank $n$ of $\PG(W,\F_{q^n})=\PG(1,q^n)$, with maximum field of linearity $\F_q$, is the number of inequivalent $\F_q$-linear blocking sets of R\'edei type of $\PG(V,\F_{q^n})=\PG(2,q^n)$ containing $L_U$. 
\end{theorem}
\begin{proof}
$\F_q$-linear blocking sets of $\PG(2,q^n)$ with more than one R\'edei line are equivalent to those defined by $\Tr_{q^n/q^m}(x)$ for some divisor $m$ of $n$, see \cite[Theorem 5]{LuPo2000}. Suppose first that $L_U$ is equivalent to $L_T$, where $T=\{(x,\Tr_{q^n/q}(x))\colon x\in \F_{q^n}\}$.
According to Theorem \ref{THMtrace} $L_T$, and hence also $L_U$, have $\ZG$-class and $\G$-class one. Proposition \ref{invar0} yields the existence of a unique point $P\in L_U$ such that $w_{L_U}(P)=n-1$. Then for each ${\bf v} \in V \setminus W$ the $\F_q$-linear blocking set defined by $\la U, {\bf v}\ra_{\F_q}$ has more than one R\'edei line, each of them incident with $P$, and hence it is equivalent to the R\'edei type blocking set obtained from $\Tr_{q^n/q}(x)$.

Now let $\cB_1=L_{V_1}$ and $\cB_2=L_{V_2}$ be two $\F_q$-linear blocking sets of R\'edei type with $\PG(W,\F_{q^n})$ the unique R\'edei line. Denote by $U_1$ and $U_2$ the $\F_q$-subspaces $W\cap V_1$ and $W\cap V_2$, respectively, and suppose $L_{U_1}=L_{U_2}$ with $\F_q$ the maximum field of linearity. Then $\cB_1$ and $\cB_2$ have $(q+1)$-secants and we have $V_1=U_1 \oplus \la {\bf u_1} \ra_{\F_q}$ and $V_2=U_2 \oplus \la {\bf u_2} \ra_{\F_q}$ for some ${\bf u_1}, {\bf u_2} \in V\setminus W$.

If $\cB_1^{\varphi_f}=\cB_2$, then \cite[Proposition 2.3]{BoPo2005} implies $V_1^f=\lambda V_2$ for some $\lambda \in \F_{q^n}^*$. Such $f\in \Gamma \mathrm{L}(3,q^n)$ has to fix $W$ and it is easy to see that $U_1^f=\lambda U_2$, i.e. $U_1$ and $U_2$ are $\Gamma \mathrm{L}(2,q^n)$-equivalent.

Conversely, if there exists $f\in \Gamma\mathrm{L}(W,\F_{q^n})$ such that $U_1^f=U_2$, then $\cB_1^{\varphi_g}=\cB_2$, where $g\in \Gamma\mathrm{L}(V,\F_{q^n})$ is the extension of $f$ mapping ${\bf u_1}$ to ${\bf u_2}$.
\end{proof}

\subsection{Class of linear sets and MRD-codes}

In \cite[Section 4]{Sh} Sheekey showed that maximum scattered linear sets of $\PG(1,q^n)$ correspond to $\F_q$-linear maximum rank distance codes (MRD-codes) of dimension $2n$ and minimum distance $n-1$, that is, a set $\cM$ of $q^{2n}$ $n\times n$ matrices over $\F_q$ forming an $\F_q$-subspace of $\F_q^{n\times n}$ of dimension $2n$ such that the non-zero matrices of $\cM$ have rank at least $n-1$.
For definitions and properties on MRD-codes we refer the reader to \cite{Delsarte} by Delsarte and \cite{Gabidulin} by Gabidulin.
For $n \times n$ matrices there are two different definitions of equivalence for MRD-codes in the literature.
The arguments of \cite[Section 4]{Sh} yield the following interpretation of the $\G$-class:

\begin{itemize}
\item $\cM$ and $\cM'$ are equivalent if there are invertible matrices $A$, $B\in\F_{q}^{n\times n}$ and a field automorphism $\sigma$ of $\F_q$ such that $A \cM^\sigma B = \cM'$, see \cite{Sh}. In this case the $\G$-class of $L_U$ is the number of inequivalent MRD-codes obtained from the linear set $L_U$.

\item $\cM$ and $\cM'$ are equivalent if there are invertible matrices $A$, $B\in\F_{q}^{n\times n}$ and a field automorphism $\sigma$ of $\F_q$  such that $A \cM^\sigma B = \cM'$, or
$A \cM^{T \sigma} B = \cM'$, see \cite{CKM2015}. In this case the number of inequivalent MRD-codes obtained from the linear set $L_U$ is between $\lceil s/2 \rceil$ and $s$, where $s$ is the $\G$-class of $L_U$.
\end{itemize}

We summarize here the known non-equivalent families of MRD-codes arising from maximum scattered linear sets.

\begin{enumerate}
\item $L_{U_1}:= \{\la (x,x^q)\ra_{\F_{q^n}} \colon x\in \F_{q^n}^*\}$ (found by Blokhuis and Lavrauw \cite{BL2000}) gives Gabidulin codes,
\item $L_{U_2}:= \{\la (x,x^{q^s})\ra_{\F_{q^n}} \colon x\in \F_{q^n}^*\}$, $\gcd(s,n)=1$ (\cite{BL2000}) gives generalized Gabidulin codes,
\item $L_{U_3}:= \{\la (x,\delta x^q + x^{q^{n-1}})\ra_{\F_{q^n}} \colon x\in \F_{q^n}^*\}$ (found by Lunardon and Polverino \cite{LP2001}) gives MRD-codes found by Sheekey,
\item $L_{U_4}:= \{\la (x,\delta x^{q^s} + x^{q^{n-s}})\ra_{\F_{q^n}} \colon x\in \F_{q^n}^*\}$, $\N(\delta)\neq 1$, $\gcd(s,n)=1$ gives MRD-codes found by Lunardon, Trombetti and Zhou in \cite{LTZ}.
\end{enumerate}

\begin{remark}
\label{nonsimpleex}
The linear sets $L_{U_1}$ and $L_{U_2}$ coincide, but when $s\notin\{1,n-1\}$, then there is no $f\in \Gamma \mathrm{L}(2,q^n)$ such that
$U_1^f=U_2$. These linear sets are of \emph{pseudoregulus type}, \cite{LuMaPoTr2014} (see also Example \ref{pseudoZG}), and in \cite{CSZ2015} it was proved that the $\G$-class of these linear sets is $\varphi(n)/2$, hence they are examples of non-simple linear sets for $n=5$ and $n>6$.
\end{remark}

It can be proved that the family $L_{U_4}$ contains linear sets non-equivalent to those from the other families. We will report on this elsewhere.

%
%
%
%

\bigskip

\noindent Bence Csajb\'ok, Giuseppe Marino and Olga Polverino\\
Dipartimento di Matematica e Fisica,\\
 Seconda Universit\`a degli Studi
di Napoli,\\
I--\,81100 Caserta, Italy\\
{\em csajbok.bence@gmail.com}, {\em giuseppe.marino@unina2.it}, {\em olga.polverino@unina2.it}

\end{document}